\documentclass{amsart}
\usepackage{amsmath}
\usepackage{amssymb}
\usepackage{amsfonts}

\newtheorem{theorem}{Theorem}
\theoremstyle{plain}

\newtheorem{definition}{Definition}
\newtheorem{example}{Example}

\newtheorem{lemma}{Lemma}

\newtheorem{proposition}{Proposition}
\newtheorem{remark}{Remark}

\numberwithin{equation}{section}
\input{tcilatex}

\begin{document}
\title[Local structure in Kaplansky algebras, I.]{On a local structure in
Kaplansky algebras. Definitions and basic properties.}
\author{Alexander A. Katz}
\address{Dr. Alexander A. Katz, Department of Mathematics and Computer
Science, St. John's College of Liberal arts and Sciences, St. John's
University, 300 Howard Ave., DaSilva Academic Center 314, Staten Island, NY
10301, USA}
\email{katza@stjohns.edu}
\date{June 26-27, 2010}
\dedicatory{Dedicated to the memory of Professor George Bachman \\
(Polytechnic University, New York, USA)}

\begin{abstract}
We introduce and study locally $AW^{\ast }$-algebras (Baer locally $C^{\ast
} $-algebras) as a locally multiplicatively convex generalization of $%
AW^{\ast }$-algebras of Kaplansky. Among other basic properties of these
algebras, it is established that:

\begin{itemize}
\item A locally $C^{\ast }$-algebra is a locally $AW^{\ast }$-algebra iff
there exists its Arens-Michael decomposition consisting entirely of $%
AW^{\ast }$-algebras;

\item A bounded part of a locally $AW^{\ast }$-algebra is an $AW^{\ast }$%
-algebra;

\item The Spectral Theorem for locally $AW^{\ast }$-algebras.
\end{itemize}
\end{abstract}

\maketitle

\section{Introduction}

The basic objects of study of what follows are Baer $^{\ast }$-algebras over 
$%
\mathbb{C}
$, which are at the same time representable as Hausdorff projective limits
of projective families of $C^{\ast }$-algebras (locally $C^{\ast }$%
-algebras).

On the one hand, the invention of Baer $^{\ast }$-algebras themselves is
connected with the development of the theory of weakly closed operator
algebras (von Neumann algebras) and the algebraic theory of complete $^{\ast
}$-regular rings. As it is known, the structure of von Neumann algebras
permits one to use geometrical and topological methods in the study of these
algebras, however, a systematic application of such methods occasionally
conceals the purely algebraic origin of some important parts of the theory
of von Neumann algebras (classification of projections, decomposition of
types, polar decomposition, existence of the dimension functions, etc.).
Thus, a tendency toward an axiomatic description of the class of operator
algebras which retain the algebraic properties of von Neumann algebras was
natural. The most important achievement in this direction was made in 1951
by Kaplansky, who in \cite{Kaplansky51} introduced a class of $AW^{\ast }$%
-algebras, most successfully realizing the idea of an algebraic description
of the non-spatial theory of von Neumann algebras (i.e. of the part of it
which is not connected with the action of the elements of the algebra as
operators on the vectors of a Hilbert space). By Kaplansky's definition, $%
AW^{\ast }$-algebras are those $C^{\ast }$-algebras whose order structure is
the same as that of von Neumann algebras. This is especially intuitively
clear for commutative algebras. In this case both, the von Neumann algebras
and the $AW^{\ast }$-algebras are $^{\ast }$-isomorphic to the algebras of
continuous complex-valued functions on an extremely disconnected Hausdorff
(Stone) compact $X$; here for a von Neumann algebra one needs to add the
condition that $X$ be Hyperstonean. The analogous connection between $%
AW^{\ast }$-algebras and von Neumann algebras is also retained in general.
The class of $AW^{\ast }$-algebras occupies the same position in the class
of $C^{\ast }$-algebras as conditionally complete vector lattices in the
class of all vector lattices, and in particular, the order properties of $%
AW^{\ast }$-algebras to a large extent determine their algebraic structure.
The Baer $^{\ast }$-rings introduced by Kaplansky in \cite{Kaplansky51},
i.e., rings with involution, in which the right annihilator of any subset is
a principal right ideal, generated by a projection, appeared as a natural
generalization of $AW^{\ast }$-algebras and complete $^{\ast }$-regular
rings. In particular, $AW^{\ast }$-algebras are precisely those Baer $^{\ast
}$-rings which are simultaneously $C^{\ast }$-algebras; complete $^{\ast }$%
-regular rings are Baer $^{\ast }$-rings which are regular in the sense of
von Neumann.

On the other hand, the Hausdorff projective limits of projective families of
Banach algebras have been studied sporadically by many authors since 1952,
when they were first introduced by Arens \cite{Arens52} and Michael \cite%
{Michael52}. The Hausdorff projective limits of projective families of $%
C^{\ast }$-algebras were first mentioned by Arens \cite{Arens52}. They have
since been studied under various names by Wenjen \cite{Wenjen58}, Sya
Do-Shin \cite{Sya59}, Brooks \cite{Brooks71}, Inoue \cite{Inoue71}, Schm\"{u}%
dgen \cite{Schmudgen75}, Phillips \cite{Phillips88a}, \cite{Phillips88b}, to
name a few. Development of the subject is reflected in the recent monograph
of Fragoulopoulou \cite{Frago05}. We will follow Inoue \cite{Inoue71} in the
usage of the name \textbf{locally }$C^{\ast }$\textbf{-algebras} for these
algebras. The Hausdorff projective limits of projective families of $W^{\ast
}$- and von Neumann algebras (under the names of locally $W^{\ast }$- and
locally von Neumann algebras resp.) were introduced and studied by
Fragoulopoulou \cite{Frago86} and Joi\c{t}a \cite{Joita99}, \cite{Joita02},
where many known results from the spatial and non-spatial theory of von
Neumann algebras were successfully generalized.

In the view of aforementioned, it is therefore interesting to introduce and
study the Hausdorff projective limits of projective families of $AW^{\ast }$%
-algebras, as well as to consider Baer locally $C^{\ast }$-algebras; to
compare these two classes, and to extend the existing theory of $AW^{\ast }$%
-algebras to the locally multiplicatively-convex case. In the present paper
(first in a series of publications under preparation) we discuss the
definitions and basic properties of these algebras which we name \textbf{%
locally }$AW^{\ast }$\textbf{-algebras}.

\begin{remark}
Elsewhere we will publish: 

$\bullet $ A description of abelian locally $AW^{\ast }$-algebras; 

$\bullet $ Hilbert modules over locally $AW^{\ast }$-algebras; 

$\bullet $ Abstract characterization of locally von Neumann algebras within
the class of locally $AW^{\ast }$-algebras; 

$\bullet $ Non-commutation integration theory for locally $AW^{\ast }$%
-algebras; 

$\bullet $ Connections between locally $AW^{\ast }$-algebras and $O^{\ast }$%
-algebras of Sarymsakov and Goldstein (see \cite{SarymsakovGoldstein76}); 

$\bullet $ Connections between locally $AW^{\ast }$-algebras and $BO^{\ast }$%
-algebras of Chilin (see \cite{Chilin81}); 

$\bullet $ Real, Jordan and Lie structures in locally $AW^{\ast }$-algebras.
\end{remark}

\section{Preliminaries}

First, we recall some basic notions on topological $^{\ast }$-algebras. A $%
^{\ast }$-algebra (or involutive algebra) is an algebra $A$ over $%
\mathbb{C}
$ with an involution 
\begin{equation*}
^{\ast }:A\rightarrow A,
\end{equation*}
such that 
\begin{equation*}
(a+\lambda b)^{\ast }=a^{\ast }+\overline{\lambda }b^{\ast },
\end{equation*}
and 
\begin{equation*}
(ab)^{\ast }=b^{\ast }a^{\ast },
\end{equation*}%
for every $a,b\in A$ and $\lambda \in $ $%
\mathbb{C}
$.

A linear seminorm $\left\Vert .\right\Vert $ on a $^{\ast }$-algebra $A$ is
a $C^{\ast }$-seminorm if it is submultiplicative, i.e. 
\begin{equation*}
\left\Vert ab\right\Vert \leq \left\Vert a\right\Vert \left\Vert
b\right\Vert ,
\end{equation*}

and satisfies the $C^{\ast }$-condition, i.e. 
\begin{equation*}
\left\Vert a^{\ast }a\right\Vert =\left\Vert a\right\Vert ^{2},
\end{equation*}
for every $a,b\in A.$ Note that the $C^{\ast }$-condition alone implies that 
$\left\Vert .\right\Vert $ is submultiplicative, and in particular 
\begin{equation*}
\left\Vert a^{\ast }\right\Vert =\left\Vert a\right\Vert ,
\end{equation*}%
for every $a\in A$ (cf. for example \cite{Frago05}).

A topological $^{\ast }$-algebra is a $^{\ast }$-algebra $A$ equipped with a
topology making the operations (addition, multiplication, additive inverse,
involution) jointly continuous. For a topological $^{\ast }$-algebra $A$,
one puts $N(A)$ for the set of continuous $C^{\ast }$-seminorms on $A.$ One
can see that $N(A)$ is a directed set with respect to pointwise ordering,
because 
\begin{equation*}
\max \{\left\Vert .\right\Vert _{\alpha },\left\Vert .\right\Vert _{\beta
}\}\in N(A)
\end{equation*}%
for every $\left\Vert .\right\Vert _{\alpha },\left\Vert .\right\Vert
_{\beta }\in N(A),$ where $\alpha ,\beta \in \Lambda ,$ with $\Lambda $
being a certain directed set.

A $C^{\ast }$-algebra is a complete Hausdorff topological algebra whose
topology is given by a single $C^{\ast }$-norm. For a topological $^{\ast }$%
-algebra $A$, and $\left\Vert .\right\Vert _{\alpha }\in N(A),$ $\alpha \in
\Lambda $, 
\begin{equation*}
\ker \left\Vert .\right\Vert _{\alpha }=\{a\in A:\left\Vert a\right\Vert
_{\alpha }=0\}
\end{equation*}%
is a $^{\ast }$-ideal in $A$, and $\left\Vert .\right\Vert _{\alpha }$
induces a $C^{\ast }$-norm (we as well denote it by $\left\Vert .\right\Vert
_{\alpha }$) on the quotient $A/\ker \left\Vert .\right\Vert _{\alpha }$, so
the completion $A_{\alpha }$ of this quotient with respect to $\left\Vert
.\right\Vert _{\alpha }$ is a $C^{\ast }$-algebra. Each pair $\left\Vert
.\right\Vert _{\alpha },\left\Vert .\right\Vert _{\beta }\in N(A),$ such
that 
\begin{equation*}
\beta \succeq \alpha ,
\end{equation*}%
$\alpha ,\beta \in \Lambda ,$ induces a natural (continuous) surjective $%
^{\ast }$-homomorphism 
\begin{equation*}
g_{\alpha }^{\beta }:A_{\beta }\rightarrow A_{\alpha }.
\end{equation*}

Let, again, $\Lambda $ be a set of indices, directed by a relation
(reflexive, transitive, antisymmetric) $"\preceq "$. Let 
\begin{equation*}
\{A_{\alpha },\alpha \in \Lambda \}
\end{equation*}%
be a family of $C^{\ast }$-algebras, and $g_{\alpha }^{\beta }$ be, for 
\begin{equation*}
\alpha \preceq \beta ,
\end{equation*}%
the continuous linear $^{\ast }$-mappings 
\begin{equation*}
g_{\alpha }^{\beta }:A_{\beta }\longrightarrow A_{\alpha },
\end{equation*}%
so that 
\begin{equation*}
g_{\alpha }^{\alpha }(x_{\alpha })=x_{\alpha },
\end{equation*}%
for all $\alpha \in \Lambda ,$ and 
\begin{equation*}
g_{\alpha }^{\beta }\circ g_{\beta }^{\gamma }=g_{\alpha }^{\gamma },
\end{equation*}%
whenever 
\begin{equation*}
\alpha \preceq \beta \preceq \gamma .
\end{equation*}%
\ Let $\Gamma $ be the collections $\{g_{\alpha }^{\beta }\}$ of all such
transformations.\ Let $A$ be a $^{\ast }$-subalgebra of the direct product
algebra%
\begin{equation*}
\dprod\limits_{\alpha \in \Lambda }A_{\alpha },
\end{equation*}%
so that for its elements 
\begin{equation*}
x_{\alpha }=g_{\alpha }^{\beta }(x_{\beta }),
\end{equation*}%
for all%
\begin{equation*}
\alpha \preceq \beta ,
\end{equation*}%
where 
\begin{equation*}
x_{\alpha }\in A_{\alpha },
\end{equation*}%
and%
\begin{equation*}
x_{\beta }\in A_{\beta }.
\end{equation*}

\begin{definition}
The $^{\ast }$-algebra $A$ above is called a \textbf{Hausdorff} \textbf{%
projective limit} of the projective family 
\begin{equation*}
\{A_{\alpha },\alpha \in \Lambda \},
\end{equation*}%
relatively to the collection%
\begin{equation*}
\Gamma =\{g_{\alpha }^{\beta }:\alpha ,\beta \in \Lambda :\alpha \preceq
\beta \},
\end{equation*}%
and is denoted by%
\begin{equation*}
\underleftarrow{\lim }A_{\alpha },
\end{equation*}%
and called the Arens-Michael decomposition of $A$.
\end{definition}

It is well known (see, for example \cite{Treves67}) that for each $\beta \in
\Lambda $ there is a natural projection 
\begin{equation*}
\pi _{\beta }:A\longrightarrow A_{\beta },
\end{equation*}%
defined by 
\begin{equation*}
\pi _{\beta }(\{x_{\alpha }\})=x_{\beta },
\end{equation*}%
and each projection $\pi _{\alpha }$ for all $\alpha \in \Lambda $ is
continuous.

\begin{definition}
A topological $^{\ast }$-algebra $A$ over $\mathbb{C}$\ \ is called a $%
LC^{\ast }$\textbf{-algebra} or \textbf{locally }$C^{\ast }$\textbf{-algebra}
if there exists a projective family of $C^{\ast }$-algebras 
\begin{equation*}
\{A_{\alpha };g_{\alpha }^{\beta };\alpha ,\beta \in \Lambda \},
\end{equation*}%
so that 
\begin{equation*}
A\cong \underleftarrow{\lim }A_{\alpha },
\end{equation*}%
i.e. $A$ is topologically $^{\ast }$-isomorphic to a projective limit of a
projective family of $C^{\ast }$-algebras, i.e. there exits its
Arens-Michael decomposition composed of $C^{\ast }$-algebras.
\end{definition}

A topological $^{\ast }$-algebra $A$ over $\mathbb{C}$ is a locally $C^{\ast
}$-algebra iff $A$ is a complete Hausdorff topological $^{\ast }$-algebra in
which topology is generated by a saturated separating family of $C^{\ast }$%
-seminorms (see \cite{Frago05} for details).

\begin{example}
Every $C^{\ast }$-algebra is a real locally $C^{\ast }$-algebra.
\end{example}

\begin{example}
A closed $^{\ast }$-subalgebra of a locally $C^{\ast }$-algebra is a locally 
$C^{\ast }$-algebra.
\end{example}

\begin{example}
The product $\dprod\limits_{\alpha \in \Lambda }A_{\alpha }$ of $C^{\ast }$%
-algebras $A_{\alpha }$, with the product topology, is a locally $C^{\ast }$%
-algebra.
\end{example}

\begin{example}
Let $X$ be a compactly generated Hausdorff space (this means that a subset $%
Y\subset X$ is closed iff $Y\cap K$ is closed for every compact subset $%
K\subset X$). Then the algebra $C(X)$ of all continuous, not necessarily
bounded complex-valued functions on $X,$ with the topology of uniform
convergence on compact subsets, is a locally $C^{\ast }$-algebra. It is
known that all metrizable spaces and all locally compact Hausdorff spaces
are compactly generated.
\end{example}

Let $A$ be a locally $C^{\ast }$-algebra. Then an element $a\in A$ is called 
\textbf{bounded}, if 
\begin{equation*}
\left\Vert a\right\Vert _{\infty }=\{\sup \left\Vert a\right\Vert _{\alpha
},\alpha \in \Lambda :\left\Vert .\right\Vert _{\alpha }\in N(A)\}<\infty .
\end{equation*}%
The set of all bounded elements of $A$ is denoted by $b(A).$

It is well-known that for each locally $C^{\ast }$-algebra $A,$ its set $b(A)
$ of bounded elements of $A$ is a $C^{\ast }$-algebra in the norm $%
\left\Vert .\right\Vert _{\infty },$ which is dense in $A$ in its topology
(see for example \cite{Frago05}).

Let again now $A$ be an associative $^{\ast }$-algebra over $%
\mathbb{C}
$. For each nonempty subset $S$ of $A$, we denote by 
\begin{equation*}
R(S)=\{x\in A:sx=\mathbf{0},\text{ }\forall s\in S\}
\end{equation*}%
\begin{equation*}
\text{(resp., }L(S)=\{x\in A:xs=\mathbf{0},\text{ }\forall s\in S\}\text{)}
\end{equation*}%
the \textbf{right }(resp. \textbf{left}) \textbf{annihilator} of $S$ in $A$.
It is clear that 
\begin{equation*}
L(S)=(R(S^{\ast }))^{\ast },
\end{equation*}%
where 
\begin{equation*}
S^{\ast }=\{s^{\ast }:s\in S\}.
\end{equation*}

A $^{\ast }$-algebra $A$ is called a \textbf{Rickart }$^{\mathbf{\ast }}$%
\textbf{-algebra}, if for each $x\in A,$ there exists a projection (i.e., a
self-adjoint idempotent) $g$ of $A$, such that 
\begin{equation*}
R(\{x\})=gA.
\end{equation*}%
Obviously, in this case 
\begin{equation*}
L(\{x\})=Ae,
\end{equation*}

for some projection $e$ in $A$.

The projections $g$ and $e$ above are uniquely determined and are called,
respectively, the \textbf{right} and the \textbf{left annihilating
projections} for $x$. The right annihilating projection for zero element is
the identity in $A$ (we denote it by $\mathbf{1}$). The projection 
\begin{equation*}
r(x)=g^{\bot }=\mathbf{1}-g
\end{equation*}%
\begin{equation*}
(l(x)=e^{\bot }=\mathbf{1}-e)
\end{equation*}%
is called the \textbf{right} (resp. \textbf{left}) support of $x$.

In the set $\mathcal{P}(A)$ of all projectors of a Rickart $^{\ast }$%
-algebra $A$ one can introduce naturally the following partial ordering: 
\begin{equation*}
e\leq f,
\end{equation*}%
if 
\begin{equation*}
e=ef,
\end{equation*}%
$e,f\in \mathcal{P}(A).$

With respect to this partial ordering, $\mathcal{P}(A)$ is a lattice (see
for example \cite{Berberian72}), which is generally not complete, and not
even $\sigma $-complete (recall here that a lattice is called complete ($%
\sigma $-complete), if any subset (countable subset) of it has a supremum
and an infimum). If $A$ is a $RC^{\ast }$-algebra or a Rickart $C^{\ast }$%
-algebra (i.e., if $A$ is simultaneously a $C^{\ast }$-algebra and a Rickart 
$^{\ast }$-algebra), the right annihilator of any countable subset of $A$ is
generated by a projection (see \cite{Berberian72}), and thus, the lattice $%
\mathcal{P}(A)$ is $\sigma $-complete. If $\mathcal{P}(A)$ is a complete
lattice, then for any nonempty subset $S$ in $A$ its right annihilator has
the form 
\begin{equation*}
R(S)=g^{\bot },
\end{equation*}%
where 
\begin{equation*}
g=\sup \{r(s):s\in S\},
\end{equation*}%
i.e., $R(S)$ is generated by a projection.

A $^{\ast }$-algebra $A$ in which the right annihilator of each nonempty
subset $S$ is generated by a projection is called a \textbf{Baer }$^{\ast }$%
\textbf{-algebra}. Thus, for a $^{\ast }$-algebra $A$, the following two
conditions are equivalent:

1). A is a Baer $^{\ast }$-algebra;

2). A is a Rickart $^{\ast }$-algebra in which the lattice of projections is
complete.

One can note that in (2) it is sufficient to require that any family of
pairwise orthogonal projections has a supremum (see \cite{Berberian72} for
details). Recall that the projections $e$ and $f$ are called \textbf{%
orthogonal} if 
\begin{equation*}
ef=\mathbf{0}.
\end{equation*}

A $^{\ast }$-subalgebra $B$ of a Baer $^{\ast }$-algebra $A$ is called a 
\textbf{Baer }$^{\ast }$\textbf{-subalgebra} if for any nonempty subset $S$
of $B$, the right annihilator of a projection for $S$ in $A$ belongs to $B$.
Obviously, a Baer $^{\ast }$-subalgebra $B$ is itself a Baer $^{\ast }$%
-algebra, and $\mathcal{P}(B)$ is a regular sublattice of $\mathcal{P}(A),$
i.e., 
\begin{equation*}
\sup_{\mathcal{P}(B)}Q=\sup_{\mathcal{P}(A)}Q,
\end{equation*}%
and 
\begin{equation*}
\inf_{\mathcal{P}(B)}Q=\inf_{\mathcal{P}(A)}Q,
\end{equation*}%
for any subset $Q$ of $\mathcal{P}(B).$ Examples of Baer $^{\ast }$%
-subalgebras in a Baer $^{\ast }$-algebra $A$ are as follows:

a). $^{\ast }$-subalgebras of the form $eAe,$ $e\in \mathcal{P}(A)$;

b). commutants 
\begin{equation*}
S^{\prime }=\{x\in A:xs=sx,\text{ }\forall s\in S\}
\end{equation*}%
of self-adjoint subsets $S$ of $A$;

c). the center of $A$ (recall that the center of a $^{\ast }$-algebra is the
intersection of all of its maximal commutative $^{\ast }$-subalgebras, i.e.
the set of elements from $A$ that are pairwise commuting with all elements
of $A$).

A Baer $^{\ast }$-algebra which is simultaneously a $C^{\ast }$-algebra is
called an \textbf{AW}$^{\ast }$\textbf{-algebra}. The definition of an $%
AW^{\ast }-$algebra was introduced by Kaplansky in \cite{Kaplansky51} in a
different but equivalent (see \cite{Kaplansky68}) form.

\begin{definition}
An $AW^{\ast }$-algebra is a $C^{\ast }$-algebra which possesses the
following properties: 

$(i).$ in the partially ordered set of all its projections, each subset of
pairwise orthogonal projections has a supremum; 

$(ii).$ any maximal commutative $^{\ast }$-subalgebra is generated in the
norm topology by its projections (i.e., coincides with the smallest closed
in its norm topology $^{\ast }$-subalgebra containing its projections).
\end{definition}

From this definition it follows that in any $AW^{\ast }$-algebra the
following property is valid:

$(iii).$\textit{\ the set of self-adjoint elements of each maximal
commutative }$^{\ast }$\textit{-subalgebra forms a conditionally complete
vector lattice with respect to the induced partial order.}

In particular, if $Z$ is a commutative $C^{\ast }$-algebra, then, by
identifying $Z$ with the $^{\ast }$-algebra $C(X)$ of all continuous
complex-valued functions on the compact set $X$, we get that $Z$ is an $%
AW^{\ast }$-algebra, if and only if $X$ is an extremely completely
disconnected compact (i.e., the closure of any open subset of $X$ is open).
Thus, the class of $AW^{\ast }$-algebras coincides with the class of $%
C^{\ast }$-algebras for which the partial order has the additional
properties $(i)$ and $(iii)$ above.

A $C^{\ast }$-subalgebra of an $AW^{\ast }$-algebra which is at the same
time a Baer $^{\ast }$-subalgebra is called an $AW^{\ast }$\textbf{%
-subalgebra}. It is clear that an $AW^{\ast }$-subalgebra is itself an $%
AW^{\ast }$-algebra and in particular, any von Neumann algebra $M$ (or any $%
W^{\ast }$-algebra) is an $AW^{\ast }$-algebra. However, not all $AW^{\ast }$%
-algebras are $W^{\ast }$-algebras, even in the commutative case. Let us
mention the following example.

\begin{example}
Let $L$ be the complete Boolean algebra of all open regular subsets of the
interval $[0,1],$ and let $X(L)$ be the Stone extremely disconnected compact
corresponding to $L$, and let 
\begin{equation*}
Z\cong C(X(L)).
\end{equation*}%
Then $Z$ is a commutative $AW^{\ast }$-algebra, but not a $W^{\ast }$%
-algebra (see \cite{Chilin85} and \cite{Pedersen79} for details)
\end{example}

\section{Baer locally $C^{\ast }$-algebras, locally Kaplansky algebras and
locally $AW^{\ast }$-algebras}

\subsection{Definitions and basic properties.}

Using various approaches to locally $C^{\ast }$-algebras, mentioned above,
one can attempt to define a locally multiplicatively-convex generalization
of $AW^{\ast }$-algebras in a few different ways- through the construction
of the projective limit, by the Baer condition on annihilators of each
subset, or through the Kaplansky's conditions imposed on the maximal
commutative $^{\ast }$-subalgebras. We show below that in the case of a
local structure in Kaplansky algebras all these approaches lead to the same
class of topological $^{\ast }$-algebras.

Let $(A,\tau )$ be a topological $^{\ast }$-algebra over $%
\mathbb{C}
.$

\begin{definition}
We call $(A,\tau )$ a \textbf{Kaplansky topological }$^{\mathbf{\ast }}$%
\textbf{-algebra} if the following two conditions are satisfied:

$(i).$ in the partially ordered set of all its projections, each subset of
pairwise orthogonal projections has a supremum;

$(ii).$ any maximal commutative $^{\ast }$-subalgebra is generated in the
topology $\tau $ by its projections (i.e., coincides with the smallest
closed in its topology $\tau $ $^{\ast }$-subalgebra containing its
projections).
\end{definition}

Now we will establish locally multiplicatively-convex versions of a few
Kaplansky's Lemmata from \cite{Kaplansky51}

\begin{lemma}
Let $B$ be a commutative locally $C^{\ast }$-algebra, where 
\begin{equation*}
B\cong \underleftarrow{\lim }B_{\alpha },
\end{equation*}%
$\alpha \in \Lambda ,$ be its Arens-Michael decomposition into projective
limit of a projective family of commutative $C^{\ast }$-algebras, and $B$ be
generated in its projective topology by its projections. Let $x\in B$, and a
positive $\epsilon $ be given. Then there exists a projection $e$ in $B$,
which is a multiple of $x$ and satisfies inequality 
\begin{equation*}
\left\Vert x-ex\right\Vert _{\alpha }<\epsilon ,
\end{equation*}%
for each $\alpha \in \Lambda .$
\end{lemma}

\begin{proof}
For a given $\alpha \in \Lambda ,$ let 
\begin{equation*}
x_{\alpha }=\pi _{\alpha }(x)\in B_{\alpha }.
\end{equation*}%
But $B_{\alpha }$ satisfies the conditions of the classical Lemma of
Kaplansky from \cite{Kaplansky51}, and thus there exists in $B_{\alpha }$ a
projection $e_{\alpha },$ such that 
\begin{equation*}
\left\Vert x_{\alpha }-e_{\alpha }x_{\alpha }\right\Vert _{\alpha }<\epsilon
.
\end{equation*}%
Let us now consider a unique element $e$ in $B$ such that 
\begin{equation*}
e=\pi _{\alpha }(e_{\alpha })
\end{equation*}%
One can easily check that $e$ is a projection in $B$ which satisfies the
conditions of the Lemma.
\end{proof}

\begin{lemma}
Let A be a locally $C^{\ast }$-algebra with its Arens-Michael decomposition 
\begin{equation*}
A\cong \underleftarrow{\lim }A_{\alpha },
\end{equation*}%
$\alpha \in \Lambda ,$ into projective limit of projective family of $%
C^{\ast }$-algebras $A_{\alpha },$ such that each of its maximal commutative 
$^{\ast }$-subalgebra is generated in its projective topology by its
projections. Let $e_{\lambda },$ $\lambda \in \Gamma ,$ be a family of
pairwise orthogonal projections in $A$ with its supremum $e$ in $A$. Then:

(a). 
\begin{equation*}
xe_{\lambda }=\mathbf{0,}
\end{equation*}%
for all $\lambda \in \Gamma ,$ implies 
\begin{equation*}
xe=\mathbf{0};
\end{equation*}

(b). 
\begin{equation*}
e_{\lambda }x=xe_{\lambda },
\end{equation*}%
for all $\lambda \in \Gamma ,$ implies 
\begin{equation*}
ex=xe.
\end{equation*}
\end{lemma}

\begin{proof}
(a). One can easily see that for a given $\alpha \in \Lambda ,$ $\pi
_{\alpha }(e_{\lambda }),$ $\lambda \in \Gamma ,$ is a family of pairwise
orthogonal projections in $A_{\alpha }$ with its supremum $\pi _{\alpha }(e)$
in $A_{\alpha }$. Thus, by applying the classical Kaplansky's Lemma from 
\cite{Kaplansky51} to $A_{\alpha },$ we get that 
\begin{equation*}
\pi _{\alpha }(x)\pi _{\alpha }(e)=\pi _{\alpha }(xe)=\mathbf{0}_{\alpha
}=\pi _{\alpha }(\mathbf{0}),
\end{equation*}%
for all $\alpha \in \Lambda ,$ thus 
\begin{equation*}
xe=\mathbf{0.}
\end{equation*}

(b). Analogously, one can easily see that for a given $\alpha \in \Lambda ,$ 
$\pi _{\alpha }(e_{\lambda }),$ $\lambda \in \Gamma ,$ is a family of
pairwise orthogonal projections in $A_{\alpha }$ with its supremum $\pi
_{\alpha }(e)$ in $A_{\alpha }$. Thus, by applying the classical Kaplansky's
Lemma from \cite{Kaplansky51} to $A_{\alpha },$ we get that 
\begin{equation*}
\pi _{\alpha }(x)\pi _{\alpha }(e)=\pi _{\alpha }(xe)=\pi _{\alpha }(ex)=\pi
_{\alpha }(e)\pi _{\alpha }(x),
\end{equation*}%
for all $\alpha \in \Lambda ,$ thus 
\begin{equation*}
xe=ex\mathbf{.}
\end{equation*}
\end{proof}

The following Theorem is valid:

\begin{theorem}
For a locally $C^{\ast }$-algebra $A$ the following conditions are
equivalent:

1). $A$ is a Baer $^{\ast }$-algebra;

2). $A$ is a Kaplansky topological $^{\ast }$-algebra;

3). there exists an Arens-Michael decomposition of $A$%
\begin{equation*}
A\cong \underleftarrow{\lim }A_{\alpha },
\end{equation*}%
such that all $A_{\alpha }$ are $AW^{\ast }$-algebras for each $\alpha \in
\Lambda ;$

4). all Arens-Michael decompositions of $A$ 
\begin{equation*}
A\cong \underleftarrow{\lim }A_{\alpha },
\end{equation*}%
are such that $A_{\alpha }$ are $AW^{\ast }$-algebras for each $\alpha \in
\Lambda .$
\end{theorem}

\begin{proof}
To prove the implication $1\Longrightarrow 4,$ let us consider $A$ to be a
Baer locally $C^{\ast }$-algebra, and 
\begin{equation*}
A\cong \underleftarrow{\lim }A_{\alpha },
\end{equation*}%
$\alpha \in \Lambda ,$ be its arbitrary Arens-Michael decomposition into
projective limit of the projective family of $C^{\ast }$-algebras $A_{\alpha
},$ $\alpha \in \Lambda ,$ which is always exists because $A$ is a locally $%
C^{\ast }$-algebra. Let us show that all $A_{\alpha },$ $\alpha \in \Lambda
, $ are Baer, and thus, all are $AW^{\ast }$-algebras. In fact, let us fix $%
\alpha \in \Lambda ,$ and let $S_{\alpha }$ be an arbitrary subset of $%
A_{\alpha }.$ Let us consider now a subset $S$ in $A$, such that 
\begin{equation*}
S=\{x\in A:\pi _{\alpha }(x)\in S_{\alpha }\},
\end{equation*}%
where $\pi _{\alpha }$ is the natural projection from $A$ to $A_{\alpha }$.
Because $A$ is a Baer $^{\ast }$-algebra, there exists a projection 
\begin{equation*}
e\in A,
\end{equation*}%
such that the right annihilator $R_{A}(S)$ of $S$ in $A$ 
\begin{equation*}
R_{A}(S)=eA.
\end{equation*}%
Now, let us consider an element 
\begin{equation*}
e_{\alpha }=\pi _{\alpha }(e)
\end{equation*}%
in $A_{\alpha }.$ It is easy to check that $e_{\alpha }$ is a projection in $%
A_{\alpha }.$ But now, one can see that the right annihilator $R_{A_{\alpha
}}(S_{\alpha })$ of $S_{\alpha }$ in $A_{\alpha }$ is such that 
\begin{equation*}
R_{A_{\alpha }}(S_{\alpha })=\pi _{\alpha }(R_{A}(S))=\pi _{\alpha }(eA)=\pi
_{\alpha }(e)\pi _{\alpha }(A)=e_{\alpha }A_{\alpha },
\end{equation*}%
and the proof of implication is completed.

The implication $4\Longrightarrow 3$ is obvious.

To prove the implication $3\Longrightarrow 1,$ let us consider a locally $%
C^{\ast }$-algebra $A$, for which there exists an Arens-Michael
decomposition 
\begin{equation*}
A\cong \underleftarrow{\lim }A_{\alpha },
\end{equation*}%
such that all $A_{\alpha }$ are $AW^{\ast }$-algebras for each $\alpha \in
\Lambda .$ Let us consider an arbitrary subset $S$ of $A$. For each $\alpha
\in \Lambda ,$ 
\begin{equation*}
S_{\alpha }=\pi _{\alpha }(S)\subset A_{\alpha }.
\end{equation*}%
Because A$_{\alpha }$ is a Baer $^{\ast }$-algebra, there exists a unique
projection $e_{\alpha }$ in $A_{\alpha },$ such that the right annihilator $%
R_{A_{\alpha }}(S_{\alpha })$ of $S_{\alpha }$ in $A_{\alpha }$ is such that 
\begin{equation*}
R_{A_{\alpha }}(S_{\alpha })=e_{\alpha }A_{\alpha },
\end{equation*}%
for each $\alpha \in \Lambda .$ Let us now consider a unique element $e$ in $%
A$, such that 
\begin{equation*}
\pi _{\alpha }(e)=e_{\alpha },
\end{equation*}%
for each $\alpha \in \Lambda .$ Because each element $e_{\alpha }$ is a
projection in $A_{\alpha },$ for each $\alpha \in \Lambda ,$ one can easily
check that $e$ is a projection in $A$, and the right annihilator $R_{A}(S)$
of $S$ in $A$ is such that 
\begin{equation*}
R_{A}(S)=eA,
\end{equation*}%
and the proof of implication is completed.

To prove the implication $3\Longrightarrow 2,$ let us consider $A$ to be a
locally $C^{\ast }$-algebra, so that there exists an Arens-Michael
decomposition of $A$%
\begin{equation*}
A\cong \underleftarrow{\lim }A_{\alpha },
\end{equation*}%
such that all $A_{\alpha }$ are $AW^{\ast }$-algebras for each $\alpha \in
\Lambda .$

On the one hand, let now $e_{\lambda },$ $\lambda \in \Gamma ,$ where $%
\Gamma $ is a directed set, be a subset of pairwise orthogonal projections
in $A$. For every fixed $\alpha \in \Lambda ,$ let us consider the set of
elements $\pi _{\alpha }(e_{\lambda }),$ $\lambda \in \Gamma $. Because all $%
e_{\lambda },$ $\lambda \in \Gamma ,$ are the subset of pairwise orthogonal
projections in $A,$ the subset of elements $\pi _{\alpha }(e_{\lambda }),$ $%
\lambda \in \Gamma ,$ will be a subset of pairwise orthogonal projections in 
$A_{\alpha }.$ But the algebra $A_{\alpha }$ is an $AW^{\ast }$-algebra, and
thus, for every $\alpha \in \Lambda ,$ there exists a projection $p_{\alpha
} $ in $A_{\alpha },$ such that 
\begin{equation*}
p_{\alpha }=\sup_{\lambda \in \Gamma }\{\pi _{\alpha }(e_{\lambda })\}.
\end{equation*}%
Let us now consider a unique element $p$ in $A$, such that 
\begin{equation*}
\pi _{\alpha }(p)=p_{\alpha },
\end{equation*}%
for each $\alpha \in \Lambda .$ One can easily check now that $p$ is a
projection in $A$, and 
\begin{equation*}
p=\sup_{\lambda \in \Gamma }\{e_{\lambda }\}.
\end{equation*}

On the other hand, let $B$ be the maximal commutative $^{\ast }$-subalgebra
of $A.$ Firstly, one can easily see that for each $\alpha \in \Lambda ,$ 
\begin{equation*}
B_{\alpha }=\pi _{\alpha }(B)
\end{equation*}%
is a maximal commutative $^{\ast }$-subalgebra of $A_{\alpha },$ and thus,
because each $A_{\alpha }$ is an $AW^{\ast }$-algebra, $B_{\alpha }$ is
generated in its norm topology by its projections. Now, on the contrary
assume that $B$ is not generated by its projections in its projective
topology. Thus, there exists an element $x$ in $B$, and an open neighborhood 
$O_{x}$ of $x$ in $B$ in the projective topology of $A$, which doesn't
contain a single projection. Let us now, for each $\alpha \in \Lambda ,$
consider a subset $\pi _{\alpha }(O_{x})$ in $B_{\alpha }.$ Because $O_{x}$
is an open neighborhood of $x$ in the projective topology of $B$, $\pi
_{\alpha }(O_{x})$ is an open neighborhood of an element $\pi _{\alpha }(x)$
in $B_{\alpha }$ in its norm topology. And because the subalgebra $B_{\alpha
}$ is generated in its norm topology by its projections, there exists a
projection in $\pi _{\alpha }(O_{x}),$ for each $\alpha \in \Lambda .$ Let
us now consider an element $e$ in $B$, such that 
\begin{equation*}
\pi _{\alpha }(e)=e_{\alpha },
\end{equation*}%
for each $\alpha \in \Lambda .$ One can easily check that $e$ is a
projection in $B$, and 
\begin{equation*}
e\in O_{x}.
\end{equation*}%
Similarly, one obtains a contradiction with the assumption that $O_{x}$ does
not contain a linear combination of a family of projections from $B$.
Obtained contradiction completes the proof of current implication.

To prove the implication $2\Longrightarrow 1,$ let us consider a locally $%
C^{\ast }$-algebra $A$ which is as well a Kaplansky topological $^{\ast }$%
-algebra. Let $S$ be any subset in $A$. By the usage of Zorn's Lemma, let us
select a maximal family of pairwise orthogonal projections $e_{\lambda
},\lambda \in \Gamma ,$ in $S$, and let a projection $e$ in $A$ be its
supremum. From Lemma 2 above it follows that $e\in R(S).$ To prove that 
\begin{equation*}
R(S)=eA,
\end{equation*}%
we take $y\in R(S)$ and have to show that 
\begin{equation*}
x=y-ey=\mathbf{0}.
\end{equation*}%
On a contrary, let us assume that 
\begin{equation*}
x\neq \mathbf{0}.
\end{equation*}%
Then, the element $xx^{\ast }$ will be as well non-zero and we can apply to
it the Lemma XX. Thus, there exists a projection $p$ in $A$, which is a
two-sided multiple of $xx^{\ast }$. Then $p$ (together with $x$) as well
belongs to $R(S)$, and 
\begin{equation*}
ex=\mathbf{0},
\end{equation*}%
implies that 
\begin{equation*}
ep=\mathbf{0},
\end{equation*}%
and 
\begin{equation*}
e_{\lambda }p=\mathbf{0},
\end{equation*}%
for all $\lambda \in \Gamma ,$ which contradicts the maximality of the
family $e_{\lambda },\lambda \in \Gamma .$

The proof of the Theorem is now completed.
\end{proof}

In a view of Theorem 1 above we now have the following definition.

\begin{definition}
A locally $C^{\ast }$-algebra which satisfies one (and thus all) of the
conditions of Theorem 1 above is called a \textbf{locally AW}$^{\mathbf{\ast 
}}$\textbf{-algebra} or \textbf{LAW}$^{\ast }$\textbf{-algebra}.
\end{definition}

As immediate corollaries of Theorem 1 we obtain the following Propositions:

\begin{proposition}
Each locally von Neumann -algebra (locally $W^{\ast }$-algebra) is a locally 
$AW^{\ast }$-algebra.
\end{proposition}

\begin{proof}
Each locally von Neumann algebra (locally $W^{\ast }$-algebra) has an
Arens-Michael decomposition consisting of all von Neumann algebras (all $%
W^{\ast }$-algebras) (see \cite{Frago86} and \cite{Joita99} for details),
and, because each von Neumann algebra ($W^{\ast }$-algebra) is an $AW^{\ast
} $-algebra (see for example \cite{Berberian72}), the statement of the
Propositions immediately follows.
\end{proof}

\begin{proposition}
In a locally AW*-algebra, the right annihilator of a right ideal is a
principal two-sided ideal generated by a central projection.
\end{proposition}

\begin{proof}
One can easily check that annihilator in question is a two-sided ideal and
is of the form $eA$, thus $e$ commutes with all elements of $A$, thus,
belongs to its center.
\end{proof}

\begin{proposition}
Each locally $AW^{\ast }$-algebra $A$ is unital.
\end{proposition}

\begin{proof}
1st Proof.

Let 
\begin{equation*}
A\cong \underleftarrow{\lim }A_{\alpha },
\end{equation*}%
$\alpha \in \Lambda ,$ be the Arens-Michael decomposition of $A$, where each
algebra $A_{\alpha }$ is an $AW^{\ast }$-algebra. But each $AW^{\ast }$%
-algebra is unital (see for example \cite{Kaplansky51}), and let its unit be
denoted by $\mathbf{1}_{\alpha }.$ Let us now consider a unique element $%
\mathbf{1}$ in A such that 
\begin{equation*}
\pi _{\alpha }(\mathbf{1})=\mathbf{1}_{\alpha },
\end{equation*}%
for all $\alpha \in \Lambda .$ One can easily check that the element $%
\mathbf{1}$ is the unit of $A$.

2nd Proof.

From Proposition XX it follows that the annihilator of zero element of $A$
is generated by a central projection, and this projection plays a role of an
identity in $A$.
\end{proof}

For what follows we introduce the following Definition:

\begin{definition}
A locally $C^{\ast }$-subalgebra of a locally $AW^{\ast }$-algebra which is
at the same time a Baer $^{\ast }$-subalgebra is called a locally $AW^{\ast
} $\textbf{-subalgebra}.
\end{definition}

Now we can formulate one more Corollary from Theorem 1.

\begin{proposition}
A locally $AW^{\ast }$-subalgebra $B$ of a locally $AW^{\ast }$-algebra $A$
is itself an $AW^{\ast }$-algebra.
\end{proposition}

\begin{proof}
Immediately follows from Theorem 1 and the analogous result about $AW^{\ast
} $-subalgebras of an $AW^{\ast }$-algebra if one considers the
Arens-Michael decomposition 
\begin{equation*}
A\cong \underleftarrow{\lim }A_{\alpha },
\end{equation*}%
$\alpha \in \Lambda ,$ of $A$ into a projective limit of a projective family
of $AW^{\ast }$-algebras $A_{\alpha }.$ But if $\pi _{\alpha }$ is the
natural projection, and 
\begin{equation*}
A_{\alpha }=\pi _{\alpha }(A),
\end{equation*}%
for $\alpha \in \Lambda ,$ then 
\begin{equation*}
\pi _{\alpha }(B)=B_{\alpha },
\end{equation*}%
$\forall \alpha \in \Lambda .$

Let us show now that $B_{\alpha }$ is an $AW^{\ast }$-subalgebra of the $%
AW^{\ast }$-algebra $A_{\alpha }$ for each $\alpha \in \Lambda .$ In is
enough to show that $B_{\alpha }$ is a Baer $^{\ast }$-algebras. In fact,
let $S_{\alpha }$ be an arbitrary subset of $B_{\alpha },$ and let 
\begin{equation*}
S=\{x\in B:\pi _{\alpha }(x)\in B_{\alpha }\},
\end{equation*}%
and thus 
\begin{equation*}
S\subset B.
\end{equation*}%
But due to the fact that $B$ is Baer, the right annihilator of $S$ in $B$ 
\begin{equation*}
R_{B}(S)=eB,
\end{equation*}%
for some projection $e$ in $B$. But then one can see that 
\begin{equation*}
e_{\alpha }=\pi _{\alpha }(e),
\end{equation*}
is a projection in $B_{\alpha },$ and one can check that the right
annihilator of $S_{\alpha }$ in $B_{\alpha }$ 
\begin{equation*}
R_{B_{\alpha }}(S_{\alpha })=e_{\alpha }B_{\alpha }.
\end{equation*}

Now, one can notice that from the fact that 
\begin{equation*}
A\cong \underleftarrow{\lim }A_{\alpha },
\end{equation*}%
$\alpha \in \Lambda ,$ and the family of algebras $A_{\alpha }$ is a
projective family with the maps $g_{\alpha }^{\beta },$ it follows that the
family $B_{\alpha }$ is as well a projective family of algebras with the
maps 
\begin{equation*}
\widehat{g}_{\alpha }^{\beta }=g_{\alpha }^{\beta }|B_{\beta },
\end{equation*}%
for each $\alpha ,\beta \in \Lambda ,$

and 
\begin{equation*}
B\cong \underleftarrow{\lim }B_{\alpha },
\end{equation*}%
$\alpha \in \Lambda .$
\end{proof}

\subsection{The center, maximal commutative $^{\ast }$-subalgebra and the
corners of a locally $AW^{\ast }$-algebra}

It was shown by Kaplansky that in an $AW^{\ast }$-algebra, its center, each
its maximal commutative $^{\ast }$-subalgebra, and each its corner
subalgebra, i.e. subalgebra of the form $eAe$, where $e$ is a projection in $%
A$, is an $AW^{\ast }$-subalgebra (see \cite{Kaplansky51}). The following
Theorem is a generalization of this result to the case of locally $AW^{\ast
} $-algebras.

\begin{theorem}
In a locally $AW^{\ast }$-algebra $A,$

(a). its center $Z$;

(b). each its maximal commutative $^{\ast }$-subalgebra $B;$

(c). each its corner subalgebra, i.e. subalgebra of the form $eAe$, where $e$
is a projection in $A$;

is a locally $AW^{\ast }$-subalgebra.
\end{theorem}

\begin{proof}
(a). Let 
\begin{equation*}
A\cong \underleftarrow{\lim }A_{\alpha },
\end{equation*}%
$\alpha \in \Lambda ,$ be the Arens-Michael decomposition of $A$ into a
projective limit of the projective family of $AW^{\ast }$-algebras $%
A_{\alpha }$ with the maps $g_{\alpha }^{\beta },$ and $Z$ be the center of $%
A.$ But, one can easily see that for each $\alpha \in \Lambda ,$ 
\begin{equation*}
Z_{\alpha }=\pi _{\alpha }(Z),
\end{equation*}%
is a center of $A_{\alpha },$ and the family of algebras $Z_{\alpha },\alpha
\in \Lambda ,$ is a projective family of commutative $AW^{\ast }$-algebras
with the maps 
\begin{equation*}
\widehat{g}_{\alpha }^{\beta }=g_{\alpha }^{\beta }|Z_{\beta },
\end{equation*}%
for each $\alpha ,\beta \in \Lambda ,$

and 
\begin{equation*}
Z\cong \underleftarrow{\lim }Z_{\alpha },
\end{equation*}%
$\alpha \in \Lambda .$

(b). Let 
\begin{equation*}
A\cong \underleftarrow{\lim }A_{\alpha },
\end{equation*}%
$\alpha \in \Lambda ,$ be the Arens-Michael decomposition of $A$ into a
projective limit of the projective family of $AW^{\ast }$-algebras $%
A_{\alpha }$ with the maps $g_{\alpha }^{\beta },$ and $B$ be the maximal
commutative *-subalgebra of $A$. But, one can easily see that for each $%
\alpha \in \Lambda ,$ 
\begin{equation*}
B_{\alpha }=\pi _{\alpha }(B),
\end{equation*}%
is a maximal commutative *-subalgebra of $A_{\alpha },$ and the family of
algebras $B_{\alpha },\alpha \in \Lambda ,$ is a projective family of
commutative $AW^{\ast }$-algebras with the maps 
\begin{equation*}
\widehat{g}_{\alpha }^{\beta }=g_{\alpha }^{\beta }|B_{\beta },
\end{equation*}%
for each $\alpha ,\beta \in \Lambda ,$

and 
\begin{equation*}
B\cong \underleftarrow{\lim }B_{\alpha },
\end{equation*}%
$\alpha \in \Lambda .$

(c). Let 
\begin{equation*}
A\cong \underleftarrow{\lim }A_{\alpha },
\end{equation*}%
$\alpha \in \Lambda ,$ be the Arens-Michael decomposition of $A$ into a
projective limit of the projective family of $AW^{\ast }$-algebras $%
A_{\alpha }$ with the maps $g_{\alpha }^{\beta },$ and $eAe$ be its corner,
with $e$ be a projection from $A$. But, one can easily see that for each $%
\alpha \in \Lambda ,$ 
\begin{equation*}
B_{\alpha }=\pi _{\alpha }(eAe)=\pi _{\alpha }(e)A_{\alpha }\pi _{\alpha
}(e)=e_{\alpha }A_{\alpha }e_{\alpha },
\end{equation*}%
where a projection $e_{\alpha }$ in A$_{\alpha }$ is such that $e_{\alpha
}=\pi _{\alpha }(e)$ , is a corner of $A_{\alpha },$ and the family of
algebras $B_{\alpha },\alpha \in \Lambda ,$ is a projective family of $%
AW^{\ast }$-algebras with the maps 
\begin{equation*}
\widehat{g}_{\alpha }^{\beta }=g_{\alpha }^{\beta }|B_{\beta },
\end{equation*}%
for each $\alpha ,\beta \in \Lambda ,$

and 
\begin{equation*}
eAe\cong \underleftarrow{\lim }e_{\alpha }A_{\alpha }e_{\alpha },
\end{equation*}%
$\alpha \in \Lambda .$
\end{proof}

\subsection{The $^{\ast }$-subalgebra of bounded elements of a locally $%
AW^{\ast }$-algebra}

As it was mentioned above, the bounded part $b(A)$ of a locally $C^{\ast }$%
-algebra $A$ is a $C^{\ast }$-subalgebra of $A$, which is dense in $A$ in
its projective topology (see for example \cite{Frago05}). When $A$ is a
locally $AW^{\ast }$-algebra, we can say a little more.

Let us recall that a famous result of Gelfand states that for any unital
Banach algebra there exists an equivalent norm, in which the norm of the
identity element is equal to $1$. Analogous result of Brooks states that for
any unital multiplicatively-covex algebra there exists a family of seminorms
that defines the same topology and such that each seminorm in that family is
equal to $1$ on the identity element (see for example \cite{Frago05} for
details). Thus, without a loss of generality in what follows we assume that
each seminorm from the separating saturated family that defines a topology
on a locally $AW^{\ast }$-algebra takes a value $1$ on the identity element.

First, we need the following Lemma:

\begin{lemma}
Any projection in a unital locally $C^{\ast }$-algebra $A$ belongs to its $%
C^{\ast }$-subalgebra $b(A)$ of bounded elements.
\end{lemma}

\begin{proof}
Let $e$ be a projection in a locally $C^{\ast }$-algebra $A$, and let 
\begin{equation*}
A\cong \underleftarrow{\lim }A_{\alpha },
\end{equation*}%
$\alpha \in \Lambda ,$ be the Arens-Michael decomposition of $A$ into a
projective limit of the projective family of $C^{\ast }$-algebras $A_{\alpha
}.$ But for each $\alpha \in \Lambda ,$ 
\begin{equation*}
\left\Vert e\right\Vert _{\alpha }=\left\Vert \pi _{\alpha }(e)\right\Vert
_{\alpha }\leq 1,
\end{equation*}%
and thus 
\begin{equation*}
\left\Vert e\right\Vert _{\infty }\leq 1<\infty .
\end{equation*}
\end{proof}

Now we are ready for our next Theorem:

\begin{theorem}
The bounded part $b(A)$\ of a locally $AW^{\ast }$-algebra is an $AW^{\ast }$%
-algebra, which is dense in $A$ in its projective topology.
\end{theorem}

\begin{proof}
Let $A$ be a locally $AW^{\ast }$-algebra, and let 
\begin{equation*}
A\cong \underleftarrow{\lim }A_{\alpha },
\end{equation*}%
$\alpha \in \Lambda ,$ be the Arens-Michael decomposition of $A$ into a
projective limit of the projective family of $AW^{\ast }$-algebras $%
A_{\alpha }.$As we already mentioned above the subalgebra $b(A)$ is a (dense
in its projective topology) $C^{\ast }$-subalgebra of $A$. Thus, it remains
for show that $b(A)$ is a Baer subalgebra of $A$. So, let a subset $S$
belongs to $b(A).$ As $S$ as well belongs to $A$, its right annihilator in $%
A $ 
\begin{equation*}
R_{A}(S)=eA,
\end{equation*}%
for some projection $e$ from $A$. From the Lemma XX above it follows that $e$
belongs to $b(A)$. Now, one can easily see that the right annihilator $%
R_{b(A)}(S)$\ of $S$ in $b(A)$ is such that%
\begin{equation*}
R_{b(A)}(S)=b(R_{A}(S))=b(eA)=e\cdot b(A).
\end{equation*}
\end{proof}

\subsection{The Spectral Theorem for locally $AW^{\ast }$-algebras}

In the present subsection we formulate and prove a version of The Spectral
Theorem for locally AW*-algebras. To begin with, let us define a notion of a
spectral family of projections.

\begin{definition}
A family of projections $e_{\lambda },$ $\lambda ,\mu \in 
\mathbb{R}
,$ from an $AW^{\ast }$-algebra $A$ is called \textbf{spectral}, if:

(i). 
\begin{equation*}
e_{\lambda }\leq e_{\mu },
\end{equation*}%
for 
\begin{equation*}
\lambda \leq \mu ;
\end{equation*}

(ii).%
\begin{equation*}
\underset{\lambda }{\sup }e_{\lambda }=\mathbf{1}\text{ and }\underset{%
\lambda }{\inf }e_{\lambda }=\mathbf{0};
\end{equation*}

(iii). 
\begin{equation*}
e_{\lambda }=\underset{\mu <\lambda }{\sup }e_{\mu }.
\end{equation*}
\end{definition}

\begin{theorem}[The Spectral Theorem for locally $AW^{\ast }$-algebras]
For any self-adjoint element $x$ in a locally $AW^{\ast }$-algebra $A$,
there exists a unique spectral family of projections $e_{\lambda },$ $%
\lambda \in \overline{%
\mathbb{R}
},$ such that $e_{\lambda }$ belongs to the maximal commutative $^{\ast }$%
-subalgebra $B$ of $A$, generated by $x$, and 
\begin{equation*}
x=\dint\limits_{-\sup_{\alpha }\left\Vert x\right\Vert _{\alpha
}}^{\sup_{\alpha }\left\Vert x\right\Vert _{\alpha }+\epsilon }\lambda
de_{\lambda },
\end{equation*}%
where $\epsilon >0,$ and by the integral we understand a limit of certain
integral sums in the projective topology of $A$.
\end{theorem}

\begin{proof}
Let $A$ be a locally $AW^{\ast }$-algebra, and%
\begin{equation*}
A\cong \underleftarrow{\lim }A_{\alpha },
\end{equation*}%
$\alpha \in \Lambda ,$ be the Arens-Michael decomposition of $A$ into a
projective limit of the projective family of $AW^{\ast }$-algebras $%
A_{\alpha }$ with the maps $g_{\alpha }^{\beta }.$ Recall that if $a$ is a
self-adjoint element from a locally $C^{\ast }$-algebra, we will denote by $%
a_{+}$ (resp. $a_{-}$) the \textbf{positive} (resp. \textbf{negative}) part
of $a$, and the following conditions are valid (see \cite{Frago05}):%
\begin{equation*}
a=a_{+}-a_{-};
\end{equation*}%
\begin{equation*}
a_{+}a_{-}=\mathbf{0}=a_{-}a_{+};
\end{equation*}%
\begin{equation*}
\left\vert a\right\vert =a_{+}+a_{-}.
\end{equation*}

Let now $x$ be a self-adjoint element from $A$. Note, that the family of
projections from $A$ defined as 
\begin{equation*}
e_{\lambda }(x)=r((\lambda 1-x)_{+}),
\end{equation*}%
is spectral, and belongs to the maximal commutative $^{\ast }$-subalgebra of 
$A$ generated by $x$. Let now for $\alpha \in \Lambda ,$ 
\begin{equation*}
\{_{\alpha }\lambda _{i}\}_{i=0}^{m},
\end{equation*}%
be an arbitrary partition of the interval 
\begin{equation*}
\lbrack -\left\Vert \pi _{\alpha }(x)\right\Vert _{\alpha },\left\Vert \pi
_{\alpha }(x)\right\Vert _{\alpha }+\epsilon ],
\end{equation*}%
for a given $\epsilon >0,$ i.e. 
\begin{equation*}
-\left\Vert \pi _{\alpha }(x)\right\Vert _{\alpha }=\text{ }_{\alpha
}\lambda _{0}<...<\text{ }_{\alpha }\lambda _{m}=\left\Vert \pi _{\alpha
}(x)\right\Vert _{\alpha }+\epsilon ,
\end{equation*}

and $\varphi _{\alpha }^{\beta }$ is a homeomorphisms from $\overline{%
\mathbb{R}
}$ to $%
\mathbb{R}
,$ such that for any $\alpha ,\beta \in \Lambda ,$ 
\begin{equation*}
(\varphi _{\alpha }^{\beta })^{-1}(_{\alpha }\lambda _{i})=\text{ }_{\beta
}\lambda _{i},
\end{equation*}%
and for each 
\begin{equation*}
i,j\in 0,1,...,m-1,
\end{equation*}%
\begin{equation*}
\frac{_{\beta }\lambda _{i+1}-\text{ }_{\beta }\lambda _{i}}{_{\alpha
}\lambda _{i+1}-\text{ }_{\alpha }\lambda _{i}}=\frac{_{\beta }\lambda
_{j+1}-\text{ }_{\beta }\lambda _{j}}{_{\alpha }\lambda _{j+1}-\text{ }%
_{\alpha }\lambda _{j}}.
\end{equation*}

Let now 
\begin{equation*}
\lambda _{0}=-\sup_{\alpha }\left\Vert \pi _{\alpha }(x)\right\Vert _{\alpha
}=-\left\Vert x\right\Vert _{\infty },
\end{equation*}%
\begin{equation*}
\lambda _{m}=\sup_{\alpha }\left\Vert \pi _{\alpha }(x)\right\Vert _{\alpha
}+\epsilon =\left\Vert x\right\Vert _{\infty }+\epsilon ,
\end{equation*}%
$\epsilon >0,$ and $\lambda _{i},$ 
\begin{equation*}
i,j=1,2,...,m-1,
\end{equation*}%
are such that for any $\alpha \in \Lambda ,$%
\begin{equation*}
\frac{_{\alpha }\lambda _{i+1}-\text{ }_{\alpha }\lambda _{i}}{\lambda
_{i+1}-\lambda _{i}}=\frac{_{\alpha }\lambda _{j+1}-\text{ }_{\alpha
}\lambda _{j}}{\lambda _{j+1}-\lambda _{j}}.
\end{equation*}

Let us now consider the integral sum 
\begin{equation*}
\sigma =\dsum\limits_{n=0}^{m}\mu _{n}(e_{\lambda _{n}}(x)-e_{\lambda
_{n-1}}(x)),
\end{equation*}%
\begin{equation*}
\lambda _{n-1}\leq \mu _{n}\leq \lambda _{n}.
\end{equation*}%
let us now take for each $\alpha ,\beta \in \Lambda ,$ $_{\alpha }\mu _{n},$
such that 
\begin{equation*}
_{\alpha }\lambda _{n-1}\leq \text{ }_{\alpha }\mu _{n}\leq \text{ }_{\alpha
}\lambda _{n},
\end{equation*}%
and 
\begin{equation*}
(\varphi _{\alpha }^{\beta })^{-1}(_{\alpha }\mu _{i})=\text{ }_{\beta }\mu
_{i},
\end{equation*}%
and 
\begin{equation*}
\frac{_{\alpha }\lambda _{n}-\text{ }_{\alpha }\mu _{n}}{\lambda _{n}-\mu
_{n}}=\frac{_{\alpha }\mu _{n}-\text{ }_{\alpha }\lambda _{n-1}}{\mu
_{n}-\lambda _{n-1}}
\end{equation*}

Let us consider in $A_{\alpha }$ for each $\alpha \in \Lambda ,$ the
integral sum 
\begin{equation*}
\pi _{\alpha }(\sigma )=\dsum\limits_{n=0}^{m}\text{ }_{\alpha }\mu _{n}(\pi
_{\alpha }(e_{\lambda _{n}}(x))-\pi _{\alpha }(e_{\lambda _{n-1}}(x))).
\end{equation*}%
One can easily see that for each $\alpha \in \Lambda ,$ 
\begin{equation*}
\left\Vert \pi _{\alpha }(x)-\pi _{\alpha }(\sigma )\right\Vert _{\alpha
}=\left\Vert \pi _{\alpha }(x-\sigma )\right\Vert _{\alpha }\leq \delta
_{\alpha },
\end{equation*}%
where 
\begin{equation*}
\delta _{\alpha }=\sup_{n}(_{\alpha }\lambda _{n}-\text{ }_{\alpha }\lambda
_{n-1}).
\end{equation*}

Therefore, for each $\alpha \in \Lambda ,$ the integral sum $\pi _{\alpha
}(\sigma )$ converges to $\pi _{\alpha }(x)$ in $A_{\alpha }$ in its norm
topology as 
\begin{equation*}
m\rightarrow \infty .
\end{equation*}%
Thus, $\sigma $ converges to $x$ in A in its projective topology, as 
\begin{equation*}
m\rightarrow \infty .
\end{equation*}
\end{proof}

\end{document}